\documentclass{amsart}

\usepackage{amssymb}

\usepackage{color}
\usepackage{graphicx}
\usepackage{enumerate}

\newtheorem{lema}{Lemma}[section]

\newtheorem{theorem}[lema]{Theorem}
\newtheorem*{Richter teo}{Theorem (Richter)}

\newtheorem{proposition}[lema]{Proposition}

\newtheorem{definicion}[lema]{Definition}

  {\par \hfill \fbox{}}
\newenvironment{demode}
  {\noindent {{\it Proof of }}}%
  {\par \hfill \fbox{}}

\def\CC{{\mathbb C}}
\def\D{{\mathcal D}}
\def\DD{{\mathbb D}}

\def\HH{{\mathcal H}}
\def\A{{\mathcal A}}

\def\LL{{\mathcal L}}
\def\M{{\mathcal M}}

\def\RR{{\mathbb R}}

\def\TT{{\mathbb T}}

\def\Re{\mathop{\rm Re}\nolimits}

\def\re{\mathop{\rm Re}\nolimits}
\def\Ker{\mathop{\rm Ker}\nolimits}
\def\dim{\mathop{\rm dim}\nolimits}

  {\par \hfill }

\begin{document}

\title[{\sc $C_0$-semigroup of 2-isometries}]
{$C_0$-semigroups of 2-isometries  and Dirichlet spaces}
\author{Eva A. Gallardo-Guti\'errez}
\address{
Universidad Complutense de Madrid e ICMAT\newline
Departamento de An\'alisis Matem\'atico,\,\newline
Facultad de Ciencias Matem\'aticas,\newline
Plaza de Ciencias 3\newline
28040, Madrid (SPAIN)}
\email{eva.gallardo@mat.ucm.es}
\author{Jonathan R. Partington}
\address{School of Mathematics,\newline
 University of Leeds,\newline
Leeds LS2 9JT, U.K.} \email{J.R.Partington@leeds.ac.uk}

\thanks{The authors are partially supported by Plan Nacional I+D grant no.
MTM2013-42105-P and MTM2016-77710-P}

\date{June 2016, Revised February 2017}


\begin{abstract}
In the context of a theorem of Richter, we establish a similarity between $C_0$-semigroups of analytic 2-isometries  $\{T(t)\}_{t\geq0}$ acting on a Hilbert space $\HH$ and the multiplication operator semigroup
$\{M_{\phi_t}\}_{t\geq 0}$ induced by $\phi_t(s)=\exp (-st)$ for $s$ in the right-half plane $\mathbb{C}_+$ acting boundedly on weighted Dirichlet spaces  on $\mathbb{C}_+$. As a consequence, we derive a connection with the right shift semigroup $\{S_t\}_{t\geq 0}$
$$
S_tf(x)=\left \{ \begin{array}{ll} 0 & \mbox { if  }0\leq x\leq t, \\
f(x-t)& \mbox { if  } x>t, \end{array} \right .
$$
acting on a weighted Lebesgue space on the half line $\mathbb{R}_+$ and address some  applications regarding the study of the invariant subspaces of $C_0$-semigroups of analytic 2-isometries.
\end{abstract}

\maketitle

\section{Introduction}

The concept of a 2-isometry was introduced by Agler in the early eighties (cf. \cite{Agler85});
this is related to
 notions   due to J. W. Helton (see \cite{He1} and \cite{He2}) and characterized in terms of  their extension properties (see \cite{Agler}).
Recall that a bounded linear operator $T$ on a separable, infinite dimensional complex Hilbert space  $\HH$ is called a
{\em 2-isometry\/} if it satisfies $$T^{*2} T^2-2T^* T+I=0,$$ where $I$ denotes the identity operator.
In addition, such  operators are called \emph{analytic} if no nonzero vector is in the range of every power of $T$.
It turns out that $M_z$, i.e. the multiplication operator by $z$, acting on the classical Dirichlet space, is a cyclic analytic 2-isometry.
But, moreover, in \cite{Ri2} (see also \cite{Ri1}) Richter proved that any cyclic analytic 2-isometry is unitarily equivalent to $M_z$ acting on a generalized Dirichlet space $D(\mu)$.

More precisely, let $\mu$ be a finite non-negative Borel measure on the unit circle $\TT$ and $D(\mu)$  the \emph{generalized Dirichlet space} associated to $\mu$, that is, the Hilbert space consisting  of analytic  functions on the unit disc $\mathbb{D}$ such that the integral
$$
\int_{\mathbb{D}} |f'(z)|^2 \left ( \int_{|\xi|=1} \frac{1-|z|^2}{|\xi-z|^2}\, d\mu(\xi)\;\right )  \frac{dm(z)}{\pi}
$$
is finite (here $dm(z)$ denotes the Lebesgue area measure in $\mathbb{D}$). Note that if $\mu=0$, the space $D(\mu)$ is defined to be the classical Hardy space $H^2$ and for non-zero, finite, non-negative Borel measures $\mu$ on $\TT$, the space $D(\mu)$ is contained in the Hardy space (see \cite[Chapter 7]{FKMR}). Then Richter's Theorem reads as follows:

\begin{Richter teo}
Let $T$ be a bounded linear operator on an infinite dimensional complex Hilbert space  $\HH$. Then the following condition are equivalent:
\begin{enumerate}
\item[$(i)$] $T$ is an analytic 2-isometry with $\dim \Ker T^*=1$,
\item[$(ii)$] $T$ is unitarily equivalent to $(M_z, D(\mu))$ for some finite non-negative Borel measure on $\TT$, where
$$
\|f\|^2_{D(\mu)}= \|f\|^2_{H^2} + \int_{\mathbb{D}} |f'(z)|^2 \left ( \int_{|\xi|=1} \frac{1-|z|^2}{|\xi-z|^2}\, d\mu(\xi)\;\right )  \frac{dm(z)}{\pi}.
$$
\end{enumerate}
\end{Richter teo}

One of the main applications of Richter's Theorem  concerns the study of the invariant subspaces for the multiplication operator $M_z$ in the spaces $D(\mu)$ and its relationship with the classical Beurling Theorem for the Hardy space $H^2$ (see \cite{Beu}).
For instance, regarding the Dirichlet space $D=D\left (\frac{|d\xi|}{2\pi}\right )$, Richter and Sundberg \cite{RiSu} proved that any closed, invariant subspace $\mathcal{M}$ under $M_z$ satisfies that $\dim \mathcal{M}\ominus z\mathcal{M}=1$.
Moreover, if $\varphi \in \mathcal{M}\ominus z\mathcal{M}$ with $\|\varphi\|_D=1$, then $|\varphi(z)|\leq 1$ for $|z|\leq 1$ and $\mathcal{M}= \varphi D(m_{\varphi})$, where $dm_{\varphi}$ is the measure on $\TT$ given by $dm_{\varphi}(\xi)= |\varphi(\xi)|^2 \frac{|d\xi|}{2\pi}$.
For general $D(\mu)$ spaces, an analogous result holds.
We refer the reader to Chapters 7 and 8 in the recent monograph \emph{``A primer on the Dirichlet space''} \cite{FKMR} for more on the subject.

Motivated by the Beurling-Lax Theorem and the work carried out by Richter, the aim of this work is taking further the study of the 2-isometries and considering $C_0$-semigroups of 2-isometric operators. In particular, we will establish a similarity between $C_0$-semigroups of analytic 2-isometries  $\{T(t)\}_{t\geq0}$ acting on a Hilbert space $\HH$ and the multiplication operator semigroup  $\{M_{\phi_t}\}_{t\geq 0}$ induced by $\phi_t(s)=\exp (-st)$ for $s$ in the right-half plane $\mathbb{C}_+$ acting boundedly on weighted Dirichlet spaces $\widetilde{\D}_{\CC_+}(\nu)$ on $\mathbb{C}_+$ (see Definition \ref{def_Dirichlet_halfplane}). As a consequence, by means of the Laplace transform, we derive a connection with the right shift semigroup $\{S_t\}_{t\geq 0}$
$$
S_tf(x)=\left \{ \begin{array}{ll} 0 & \mbox { if  }0\leq t\leq x, \\
f(x-t)& \mbox { if  } x>t, \end{array} \right .
$$
acting on a weighted Lebesgue space on the half line $\mathbb{R}_+$.  Finally, some applications regarding the study of the invariant subspaces of $C_0$-semigroups of analytic 2-isometries are also discussed in Section \ref{Sec 3}.

\section{$C_0$-semigroups of analytic 2-isometries}

First, we introduce some basic concepts and terminology regarding $C_0$-semigroups of bounded linear operators. For more on this topic, we refer the reader to the
Engel--Nagel monograph \cite{EnNa}.

A $C_0$-semigroup  $\{T(t)\}_{t\geq0}$ of operators on a Hilbert space $\HH$ is a family of bounded linear operators on $\HH$ satisfying the functional equation
\begin{equation*}
\begin{cases}
\begin{array}{lll}
T(t+s)=T(t)T(s)		&\hbox{for all }t,s\geq0,\\
T(0)=I,
\end{array}
\end{cases}
\end{equation*}
and such that $T(t)\to I$ in the strong operator topology as $t\to0^+$. Given a $C_0$-semigroup $\{T(t)\}_{t\geq0}$, there exists a closed and densely defined linear operator $A$ that determines the semigroup uniquely, called the generator of $\{T(t)\}_{t\geq0}$, defined by means of
\begin{equation*}
Ax:=\lim_{t\to0^+}\frac{T(t)x-x}{t},
\end{equation*}
where the domain $D(A)$ of $A$ consists of all $x\in\HH$ for which this limit exists (see \cite[Chapter II]{EnNa}, for instance). Although the generator is, in general, an unbounded operator, it plays an important role in the study of a $C_0$-semigroup, reflecting many of its properties.

However, if 1 is in the resolvent of $A$, that is, in the set
$$\rho(A)=\{\lambda \in \mathbb{C}:\; (A-\lambda I): D(A)\subset \HH\to \HH \mbox{ is bijective}\},$$
then $(A-I)^{-1}$ is a bounded operator on $\HH$ by the Closed Graph Theorem, and the
Cayley transform of $A$ defined by
\begin{equation*}\label{eq:cogenerator}
V:=(A+I)(A-I)^{-1}
\end{equation*}
is a bounded operator on $\HH$, since $V-I=2(A-I)^{-1}$. Therefore $V$ determines the semigroup uniquely, since $A$ does. This operator is called the \emph{cogenerator} of the $C_0$-semigroup $\{T(t)\}_{t\geq0}$. Observe that 1 is not an eigenvalue of $V$.

Recall that if $A$ is a closed operator,  then the \emph{spectral bound} $s(A)$ of $A$ is defined by
$$
s(A):=\sup \{\Re \lambda:\; \lambda \in \sigma(A)\},
$$
where $\sigma(A)= \mathbb{C}\setminus \rho(A)$ is the spectrum of $A$, and in case that $A$ is the generator of a $C_0$-semigroup, then $s(A)$ is always dominated by the \emph{growth bound of the semigroup}, that is,
$$
-\infty\leq s(A)\leq w_0=\inf \left \{w\in \mathbb{R}:\; \begin{array}{ll}
\mbox{ there exists } M_w\geq 1 \mbox{ such that }\\
\|T(t)\|\leq M_w \; e^{wt} \mbox{ for all } t\geq 0
\end{array}\right \}.
$$
Indeed, if $r(T(t))$ denotes the spectral radius of $T(t)$, it follows that $w_0=\frac{1}{t} \log r(T(t))$ for each $t>0$ (see \cite[Section 2, Chapter IV]{EnNa}, for instance). The following lemma will be useful in the context of our main result later.

\begin{lema}\label{lemma:2-isometric generator}
Let $\{T(t)\}_{t\geq 0}$ be a $C_0$-semigroup on a separable, infinite dimensional complex Hilbert space $\HH$ consisting of 2-isometries and $A$ its generator. Then $1  \in \rho(A)$ and therefore, the cogenerator $V$ of  $\{T(t)\}_{t\geq 0}$ is well-defined.
\end{lema}

\begin{proof}
By induction it follows that, for any $n\geq 1$ and $t\geq 0$, $T(t)$ satisfies
\[
T(t)^{*n} T(t)^n-nT(t)^* T(t)+(n-1)I=0,
\]
and so
$$
\|T(t)^n x\|^2= n \|T(t) x\|^2- (n-1) \|x\|^2
$$
for $x\in \HH$. From here, it follows that $\|T(t)^n\|\leq C\; \sqrt{n}$, where $C$ is a constant independent of $n$, and therefore the spectral radius $r(T(t))\leq 1$ for any $t$. Therefore, $s(A)\leq 0$; and therefore $1\in \rho(A)$.
\end{proof}

The next result consists of a particular instance of \cite[Theorem 1]{JPPW}, where $C_0$-semigroups of hypercontractions are considered. We state it for $C_0$-semigroups of 2-isometries and include its proof for the sake of completeness.

\begin{proposition}\label{prop:2 isometry}
Let $\{T(t)\}_{t\geq 0}$ be a $C_0$-semigroup on a separable, infinite dimensional complex Hilbert space $\HH$. Then the following conditions are equivalent:
\begin{enumerate}[(i)]
\item $T(t)$ is a 2-isometry for every $t\geq 0$.
\item The mapping $t\in \mathbb{R}_+ \mapsto \|T(t)x\|^2$ is affine for each $x \in \HH$.
\item $\re \langle A^2y, y \rangle + \|Ay\|^2=0  \qquad (y \in \D(A^2))$.
\item The cogenerator $V$ of $\{T(t)\}_{t\geq 0}$ exists and is a 2-isometry.
\end{enumerate}
\end{proposition}

\begin{proof}
$(i) \iff (ii)$: If each $T(t)$ is a 2-isometry, then for $t \ge 0$ and $\tau>0$ we have
\[
\langle T(t+2\tau)x,T(t+2\tau)x \rangle - 2 \langle T(t+\tau)x,T(t+\tau)x \rangle +
\langle T(t)x,T(t)x \rangle = 0,
\]
so that
\begin{equation}\label{eq:convex}
\|T(t+\tau)x\|^2 = \frac12 (\|T(t)x\|^2 + \|T(t+2\tau)x\|^2).
\end{equation}
Since $t\in \mathbb{R}_+ \to \|T(t)x\|^2$ is continuous, the mapping is affine.

Conversely, taking $t=0$ we see that \eqref{eq:convex} implies that $T(\tau)$ is a 2-isometry.\\

$(ii) \iff (iii)$: For $t>0$ we calculate the second derivative of the function $g: t \mapsto \|T(t)y\|^2$ for
$y \in \D(A^2)$.
We have
\begin{eqnarray*}
g''(t)&=&\frac{d^2}{dt^2}\langle T(t)y, T(t)y\rangle \\
&=& \langle A^2 T(t)y,T(t)y \rangle
+ 2 \langle AT(t)y, AT(t)y \rangle + \langle T(t)y,A^2T(t)y \rangle.
\end{eqnarray*}
For $g$ affine, $g''$ is zero, and Condition (iii) follows on letting $t \to 0$. Conversely,
Condition (iii) implies Condition (ii) for $y \in \D(A^2)$, and hence for all $y$ by density.\\

$(iii) \iff (iv)$:
We calculate
\[
\langle ( I-2V^*V+V^{*2}V^2)x,x\rangle
\]
for $x=(A-I)^2y$ (note that $(A-I)^{-2}:H \to H$ is
defined everywhere and has dense range).
We obtain
\begin{eqnarray*}
&& \langle (A-I)^2y,(A-I)^2y \rangle
- 2\langle (A^2-I)y, (A^2-I)y \rangle
+ \langle (A+I)^2y,(A+I)^2y \rangle \\
&& \qquad =
4\langle A^2 y,y \rangle + 8 \langle Ay, Ay \rangle
+ 4 \langle y, A^2y \rangle  .
\end{eqnarray*}
Thus $V$ is a 2-isometry if and only if Condition (iii) holds.
\end{proof}

Before stating  the main result of the section, let us introduce the following definition.

\begin{definicion}\label{def_Dirichlet_halfplane}
Let $\nu$ be a finite positive Borel measure supported on the imaginary axis. The Dirichlet space  $\widetilde{\D}_{\CC_+}(\nu)$ is defined as the space of analytic functions $F$ on right half-plane $\CC_+$ such that
\[
\|F\|^2= |F(1)|^2+ \frac{1}{\pi} \int_{\CC_+} |F'(s)|^2 \left (x+ \frac{1}{\pi} \int_{-\infty}^{\infty} \frac{x}{x^2+(y-\tau)^2}\, d\nu(\tau) \right) \, dx\, dy<\infty
\]
where $s=x+iy$.
\end{definicion}

The spaces $\widetilde{\D}_{\CC_+}(\nu)$ arise, in a natural way, when we analyze $C_0$-semigroups of analytic 2-isometries in Hilbert spaces, as it is stated in our main result:

\begin{theorem}\label{thm:main}
Let $\{T(t)\}_{t\geq 0}$ be a $C_0$-semigroup on a separable, infinite dimensional complex Hilbert space $\HH$ consisting of analytic 2-isometries for every $t>0$ such that
\begin{equation}\label{equ1}
\dim \bigcap_{t>0} \ker \left (T^*(t)-e^{-t}\,I\right ) =1.
\end{equation}
Then there exists a finite positive Borel measure $\nu$ supported on the imaginary axis such that
$\{T(t)\}_{t\geq 0}$ is similar to the semigroup of multiplication operators induced by $\exp(-ts)$ acting on the space $\widetilde{\D}_{\CC_+}(\nu)$. Moreover, if the multiplication operators induced by $\exp(-ts)$ act continuously for every $t>0$ on a Dirichlet space $\widetilde{\D}_{\CC_+}(\tilde{\nu})$ where $\tilde{\nu}$ is a finite positive Borel measure supported on the imaginary axis, then the corresponding semigroup consists of analytic 2-isometries and satisfies (\ref{equ1}).
\end{theorem}

Before proceeding further, let us remark that our main result yields similarity for the semigroup $\{T(t)\}_{t\geq 0}$ because of the definition of the norm  in $\widetilde{\D}_{\CC_+}(\nu)$. In addition, as we shall see later, condition (\ref{equ1}) is a way of expressing the property that  $\dim \ker V^*=1$, where $V$ is the  cogenerator of the semigroup $\{T(t)\}_{t\geq 0}$.

In order to prove Theorem \ref{thm:main}, we need the following auxiliary results.

\begin{proposition}\label{prop3:analytic}
Let $\{T(t)\}_{t\geq 0}$ be a $C_0$-semigroup on a separable, infinite dimensional complex Hilbert space $\HH$ consisting of analytic 2-isometries. Then the cogenerator $V$ is an analytic 2-isometry.
\end{proposition}

\begin{proof}
First, we observe that $V$ is well-defined by Lemma \ref{lemma:2-isometric generator} and, it is a 2-isometry by Proposition \ref{prop:2 isometry}.
So, we are required to show that $V$ is analytic.

The Wold Decomposition Theorem for 2-isometries (see \cite{Olof}, for instance), yields that $V$ can be decomposed as $V=S\oplus U$ with respect to $\HH= \HH_1\oplus \HH_2$, where $U$ is the unitary part on $\HH_2=\bigcap_{n} V^n \HH$ and $S$ is an analytic 2-isometry. We will show that $U=0$.

Let us assume, on the contrary, that $U\neq 0$.

First, we observe that since 1 is not an eigenvalue of $V$, the generator $A$ of the semigroup $\{T(t)\}_{t\geq 0}$ may be expressed as the (possibly) unbounded operator $$(V+I)(V-I)^{-1}.$$  Moreover, since $T(t)$ commutes with $(A-I)^{-1}$ and hence with $V$, it holds that $\HH_2$ is invariant under  $T(t)$ for every $t\geq 0$. In addition, the generator $B$ of the restricted semigroup $\{T(t)_{|_{\HH_2}}\}_{t\geq 0}$ is the restriction  of $A$ to the $D(A)\cap \HH_2$ (see \cite[Ch.\ 2, Sec.\ 2]{EnNa}, for instance); and the cogenerator is $U$.

Now, taking into account the fact that $U$ is unitary, one deduces that $B$ is skew-adjoint (i. e., $B^\star=-B$). Then the restriction of $T(t)$ to $\HH_2$ is unitary for every $t\geq 0$ and, therefore,  every vector in $\HH^2$ is in the range of (powers of) $T(t)$. Since $T(t)$ is analytic, it follows that $\HH^2=\{0\}$,
a contradiction.
Hence, $U=0$ and the proof  is completed.
\end{proof}

\begin{lema}\label{lemma:eigenvectors}
Let $\{T(t)\}_{t\geq 0}$ be a $C_0$-semigroup on a separable, infinite dimensional complex Hilbert space $\HH$ and $A$ its generator. The following conditions are equivalent:
\begin{enumerate}
\item[(1)] $Ax_0=-x_0$  for some $x_0\in D(A)$.
\item[(2)] $T(t) x_0= e^{-t} x_0$ for all $t\geq 0$ and $x_0\in D(A)$.
\end{enumerate}
In addition, if $1\in \rho(A)$ and $V$ is the cogenerator, any of the previous conditions is equivalent to
\begin{enumerate}
\item[(3)] $Vx_0=0$  for some $x_0\in D(A)$.
\end{enumerate}
\end{lema}

Note that the equivalence between (1) and (2) in Lemma \ref{lemma:eigenvectors} just follows from the relationship between the eigenspaces of $A$ and the semigroup $\{T(t)\}_{t\geq 0}$, that is, $$\Ker (\lambda I-A)=\displaystyle \bigcap_{t\geq 0} \Ker (e^{\lambda t}-T(t)),$$ with $\lambda \in \mathbb{C}$ (see \cite[Corollary 3.8, Section IV]{EnNa}, for instance). The last statement follows from the definition of $V$.

We are now in position to prove Theorem \ref{thm:main}.\\

\begin{demode}\emph{Theorem \ref{thm:main}}
Assume that $\{T(t)\}_{t>0}$ consists of analytic 2-isometries. Let $V$ denote its cogenerator; this is well-defined by Lemma \ref{lemma:2-isometric generator}, and it is an analytic 2-isometry by Proposition \ref{prop3:analytic}.

In addition, the hypotheses $\dim \bigcap_{t>0} \ker \left (T^*(t)-e^{-t}\,I\right ) =1$ along with Lemma \ref{lemma:eigenvectors} applied to the adjoint
semigroup $\{T^*(t)\}_{t\geq 0}$, yields that $\dim \Ker V^*=1$.

By means of Richter's Theorem, it follows that $V$ is similar to $M_z$ acting on the space $D(\mu)$ for some finite non-negative Borel measure $\mu$ on $\TT$ considered with the equivalent norm
\begin{eqnarray}
\|f\|^2_{D(\mu)}&\approx&\displaystyle  |f(0)|^2+ \int_{\mathbb{D}} |f'(z)|^2 \left ( \int_{|\xi|=1} \frac{1-|z|^2}{|\xi-z|^2}\, d\mu(\xi)\;\right )  \frac{dm(z)}{\pi} \nonumber\\
&=&\displaystyle   |f(0)|^2 + \int_{\mathbb{D}} |f'(z)|^2 P_{\mu}(z)  \frac{dm(z)}{\pi}. \label{eq1}
\end{eqnarray}
Observe that the similarity is the price paid when we consider the equivalent norm. Hence, for any $t\geq 0$, it follows that $T(t)$ is unitarily equivalent to the multiplication operator induced by $\exp(-t(1+z)/(1-z))$ on $D(\mu)$.
Now, we migrate to the right half-plane $\CC_+=\{\Re s>0\}$ applying the change of variables $s=(1+z)/(1-z)$, or $z=(s-1)/(s+1)$.


 First, we observe that
$$P_{\mu}\left (\displaystyle \frac{s-1}{s+1}\right )=\displaystyle \int_{|\xi|=1} \displaystyle \frac{1-|\frac{s-1}{s+1}|^2}{|\xi-\frac{s-1}{s+1}|^2}\, d\mu(\xi)\qquad (s\in \CC_+) $$
is a positive harmonic function in $\CC_+$;
so there exists a non-negative constant $\rho$ and a finite positive Borel measure $\nu$ supported on the imaginary axis such that
\begin{equation}\label{eq:eq}
P_{\mu}\left (\displaystyle \frac{s-1}{s+1}\right )= \rho\, x+ \frac{1}{\pi} \int_{-\infty}^{\infty} \frac{x}{x^2+(y-\tau)^2}\, d\nu(\tau),\qquad (s=x+iy)
\end{equation}
(see \cite[Exercise 6, p. 134]{Hoffman}, for instance).

We can express $\nu$ in terms of $\mu$, since with $\xi=(u-1)/(u+1)$ for $u=i\tau \in i\RR$, we have
\begin{eqnarray*}
P_{\mu}\left (\displaystyle \frac{s-1}{s+1}\right )&=& \mu(1)\frac{|s+1|^2-|s-1|^2}{|(s+1)-(s-1)|^2} + \int_{\xi \in \TT \setminus \{1\}}
\frac{(|s+1|^2-|s-1|^2) |u+1|^2} {|(u-1)(s+1)-(u+1)(s-1)|^2}  \, d\mu(\xi)   \\
&=&\mu(1)x + \int_{\xi \in \TT \setminus \{1\}}
\frac{x |u+1|^2} {|u-s|^2}  \, d\mu(\xi),   \\
&=&
 \mu(1)x + \int_{\xi \in \TT \setminus \{1\}}
\frac{x (1+\tau^2)} {x^2+(y-\tau)^2}  \, d\mu(\xi),
\end{eqnarray*}
 where $s=x+iy \in \CC_+$.
So in \eqref{eq:eq} we have
\begin{equation}\label{eq:munu}
\rho=\mu(1) \qquad \hbox{and} \qquad \dfrac{d\nu(\tau)}{\pi(1+\tau^2)}=  d\mu(\xi) .
\end{equation}

Then, upon applying the change of variables $s=(1+z)/(1-z)$ in (\ref{eq1}), we deduce that $T(t)$ is similar to the multiplication operator induced by $\exp(-ts)$ acting on the space $\widetilde{\D}_{\CC_+}(\nu)$ consisting  of analytic functions $F$ on $\CC_+$ such that
\begin{equation}\label{eq2}
\frac{1}{\pi} \int_{\CC_+} |F'(s)|^2 \left (x+ \frac{1}{\pi} \int_{-\infty}^{\infty} \frac{x}{x^2+(y-\tau)^2}\, d\nu(\tau) \right) \, dx\, dy<\infty,
\end{equation}
where $s=x+iy$ and $F(s)=f(z)$. This proves the first half of Theorem \ref{thm:main}.


In order to conclude the proof, let us assume that the multiplication operators induced by $\exp(-ts)$ act continuously for every $t>0$ on a Dirichlet space $\widetilde{\D}_{\CC_+}(\tilde{\nu})$ where $\tilde{\nu}$ is a finite positive Borel measure supported on the imaginary axis.
Reversing the steps above and taking  into account the fact that (\ref{eq:munu})  defines a measure  $\tilde{\mu}$ on $\TT$, where $\tilde{\mu}(1)=\tilde{\nu}(0)=\rho$, we deduce that  the given semigroup is similar to 
the semigroup of multiplication operators  induced by $\phi_t(z)=\exp(-t(1+z)/(1-z))$ on $D(\tilde{\mu})$.
Since the cogenerator of such a semigroup is $M_z$, which is a 2-isometry, it follows by Proposition \ref{prop:2 isometry} that $\{M_{\phi_t}\}_{t\geq 0}$ consists of  2-isometries.

It remains to show that $M_{\phi_t}$ is analytic for every $t > 0$. If not, then there are a $t_0>0$ and a
$F \in \widetilde{\D}_{\CC_+}(\tilde{\nu})$ such that the function $s \mapsto e^{nt_0s}F(s)$ lies in
$\widetilde{\D}_{\CC_+}(\tilde{\nu})$ for $n=1,2,3,\ldots$.

In particular,
\[
\int_{\CC_+} |(F(s)e^{nt_0s})'|^2 x \, dx \, dy < \infty.
\]
Transferring to the disc by letting $s=(1+z)/(1-z)$ and $F(s)=f(z)$, we have
\[
\int_\DD \left| [f(z)\exp(nt_0(1+z)/(1-z))]' \right|^2 \frac{1-|z|^2}{|1-z|^2} \, dA(z) < \infty,
\]
so that the function $z \mapsto f(z)\exp(nt_0(1+z)/(1-z))$ lies in the weighted Dirichlet space
$D(\delta_1)$ corresponding to a Dirac measure at $1$, and hence in $H^2(\DD)$,
by  \cite[Thm. 7.1.2]{FKMR}.
We conclude that $f$ is identically zero, since no nontrivial $H^2$ function can be divisible
by an arbitrarily large power of a nonconstant inner function.  Hence the analyticity is also established.

\end{demode}

\subsection*{A connection with the right-shift semigroup in weighted $L^2(\mathbb{R}_+)$}
Now, by means  of the Laplace transform, we will establish a connection of $C_0$-semigroups
of analytic 2-isometries  $\{T(t)\}_{t\geq0}$ acting on a Hilbert space $\HH$ and the  the right shift semigroup $\{S_t\}_{t\geq 0}$
$$
S_tf(x)=\left \{ \begin{array}{ll} 0 & \mbox { if  }0\leq x\leq t, \\
f(x-t)& \mbox { if  } x>t, \end{array} \right .
$$
acting on a weighted Lebesgue space on the half line $\mathbb{R}_+$.

First, let us begin by recalling a result asserting that for each $\alpha>-1$, a function $G$ analytic
in $\CC_+$ belongs to the {\it weighted Bergman space}  $\A^2_\alpha(\CC_+)$, that is, the space consisting of analytic functions on $\CC_+$ for which
$$
\| G \|_{\A^2_\alpha(\Pi^+)}^2  = \int_{\CC_+} |G(x+iy) |^2  \, x^{\alpha}\, dx\, dy< \infty\,,
$$
if and only if it has the form
$$
G(s):= \LL g(s)=\int_0^\infty e^{-st}\, g(t)\, dt\,, \qquad  s\in\CC_+\,,
$$
where $g$ is a measurable function on
$\RR^+$ with
$$
\int_0^\infty |g(t)|^2 \, t^{-1-\alpha}\, dt < \infty\,.
$$
Moreover,
$$
\|G\|^2_{\A^2_\alpha(\CC^+)} = \frac{\pi\,\Gamma(1+\alpha)}{2^\alpha}
\int_0^\infty |g(t)|^2 \, t^{-1-\alpha}\, dt\,,
$$
(see \cite{DP94} or \cite[Theorem 1]{DGM05}, for instance). In other words, the Laplace transform is an isometric isomorphism between $\A^2_\alpha(\CC_+)$
and $L^2(\RR_+, \left ( \frac{\pi\,\Gamma(1+\alpha)}{2^\alpha} \right )^{1/2} t^{-1-\alpha}\, dt)$. Hence, by means of a density argument, and taking $\alpha=1$, it follows that for any $F\in \widetilde{\D}_{\CC_+}(\nu)$, there exists $f\in L^2(\RR_+)$ (unique in the usual sense of equivalence classes), such that
\begin{enumerate}
\item[$(i)$] $\LL (tf(t))=F'(s)$,
\item [$(ii)$]
$$
\frac{1}{\pi} \int_{\CC_+} |F'(s)|^2 \, x\,  dx\, dy= \frac{1}{2} \, \int_0^{\infty} |f(t)|^2 dt,\qquad (s=x+iy),
$$
which corresponds to the first sum in (\ref{eq2}); and
\item[$(iii)$]
$$
\frac{1}{\pi^2} \int_{\CC_+} |F'(s)|^2   \frac{x}{x^2+(y-\tau)^2}\, dx\, dy= \frac{1}{2\pi} \, \int_0^{\infty} \left |\int_0^t  u\,f(u) e^{-i\tau u} \, du \right |^2  \frac{dt}{t^2},\quad (s=x+iy).
$$
\end{enumerate}
These three items along with the fact that for any $t\geq 0$,  $T(t)$ is similar to the multiplication operator induced by $\exp(-ts)$ acting on the space $\widetilde{\D}_{\CC_+}(\nu)$  yields,
by means of the Laplace transform, that $\{T(t)\}_{t\geq 0}$ is transformed to the right-shift semigroup $\{S_{t}\}_{t\geq 0}$ acting on the Hilbert space $\mathfrak{H}$ which consists of functions $f$ defined on $\mathbb{R}_+$ such that
$$
\int_0^{\infty} |f(t)|^2 \, dt+ \int_0^{\infty} \int_{-\infty}^{\infty} \left | \int_0^t f(u)e^{-i\tau u} u\,du\right |^2 \, d\nu (\tau)\;  \frac{dt}{t^2}< \infty.
$$
\\


\section{A final remark on invariant subspaces of $C_0$-semigroups of analytic 2-isometries}\label{Sec 3}

As an application of our main result, we deal with the study of the lattice of the closed invariant subspaces of a $C_0$-semigroup  $\{T(t)\}_{t\geq0}$ of analytic 2-isometries.

Here we shall use the following result from \cite[Thm.~7.1]{Ri2} and \cite[Thm.~3.2]{RiSu}.

\begin{theorem}\label{thm:richter12}
Let $\M$ be a non-zero invariant subspace of $(M_z,D(\mu))$. Then $\M=\phi D_{\mu_\phi}$ where $\phi \in \M \ominus z\M$ is a multiplier
of $D(\mu)$ and $d\mu_\phi=|\phi|^2 d \mu$.
\end{theorem}

 In the continuous case we have the following result:

\begin{theorem}
Let $\{T(t)\}_{t \ge 0}$ denote the semigroup of multiplication operators induced by $\exp(-ts)$ on the space
$\widetilde{\D}_{\CC_+}(\nu)$, as in Theorem~\ref{thm:main}, and let $\M$ be a non-zero closed subspace invariant under all
the operators $T(t)$. Then there is a function $\psi \in \M$ such that $\M=\psi \widetilde{\D}_{\CC_+}(\nu_\psi)$.
\end{theorem}

\begin{proof}
If $\M$ is invariant under the semigroup, then it is also invariant under the
cogenerator $V$, and after transforming to the disc
as in the proof of Theorem \ref{thm:main},
we may apply Theorem \ref{thm:richter12}.

Note that under the equivalence between $D(\mu)$ and $\widetilde{\D}_{\CC_+}(\nu)$, as detailed in \eqref{eq:eq} and \eqref{eq:munu},
the subspace $\phi D_{\mu_\phi}$ maps to a space $ \psi \widetilde{\D}_{\CC_+}(\nu_\psi)$,
where $\psi(s)=\phi((s-1)/(s+1))$ and $d\nu_\psi=|\psi|^2 d\nu$.
\end{proof}

\section*{Acknowledgements}

The authors thank the referees for a careful reading of the manuscript as well as for their suggestions, which improved the final version of the submitted manuscript.

\end{document}